\documentclass[11pt]{amsart}	 

\usepackage[english]{babel}
\usepackage[utf8]{inputenc}
\usepackage[T1]{fontenc}
\usepackage{a4wide}

\usepackage{amsmath,amsfonts,amsthm} 
\usepackage[svgnames]{xcolor} 

\usepackage{graphicx}
\usepackage{url}
\usepackage{enumerate}
\usepackage{tikz}
\usepackage{stmaryrd}
\usepackage{color}
\usepackage{pdfpages}

\newtheorem{theorem}{Theorem}
\newtheorem{proposition}[theorem]{Proposition}

\newtheorem{lemma}[theorem]{Lemma}
\newtheorem{corollary}[theorem]{Corollary}

\theoremstyle{definition}
\newtheorem{definition}[theorem]{Definition}

\theoremstyle{remark}

\newtheorem{rem}{Remark}

\renewcommand{\(}{\left(}
\renewcommand{\)}{\right)}
\newcommand{\rcp}[1]{\frac{1}{#1}}
\newcommand{\bth}{\begin{theorem}}
    \newcommand{\eth}{\end{theorem}}
\newcommand{\bl}{\begin{lemma}}
    \newcommand{\el}{\end{lemma}}
\newcommand{\bp}{\begin{proposition}}
    \newcommand{\ep}{\end{proposition}}
\newcommand{\bdf}{\begin{definition}}
    \newcommand{\edf}{\end{definition}}
\newcommand{\brem}{\begin{rem}}
    \newcommand{\erem}{\end{rem}}
\newcommand{\bcor}{\begin{corollary}}
    \newcommand{\ecor}{\end{corollary}}
\newcommand{\bpf}{\begin{proof}}
    \newcommand{\epf}{\end{proof}}
\newcommand{\eps}{\varepsilon}

\newcommand{\Sc}{\scriptstyle}

\newcommand{\simn}{\underset{n\rightarrow \infty}\sim}

\begin{document}
    
    \title{On the Number of Increasing Trees with Label Repetitions}
    \date{\today}
    \thanks{This work was partially supported by the ANR project {\em MetACOnc}
        ANR-15-CE40-0014, by the Austrian Science Fund (FWF) grant SFB F50-03, 
        the PHC Amadeus project 39454SF, and the National Research Foundation of South Africa, grant number 96236.}
    
    \author{Olivier Bodini}
    \address{Laboratoire d'Informatique de Paris-Nord,
        CNRS UMR 7030 - Institut Galil\'ee - Universit\'e Paris-Nord,
        99, avenue Jean-Baptiste Cl\'ement, 93430 Villetaneuse, France.} 
    \email{Olivier.Bodini@lipn.univ-paris13.fr}
    
    \author{Antoine Genitrini}
    \address{Sorbonne Universit\'e,	CNRS,
        Laboratoire d'Informatique de Paris 6 -LIP6- UMR 7606, F-75005 Paris, France.}
    \email{Antoine.Genitrini@lip6.fr}
    
    \author{Bernhard Gittenberger}
    \address{Department of Discrete Mathematics and Geometry, 
        Technische Universit\"at Wien, Wiedner Hauptstra\ss e 8-10/104, 1040 Wien, Austria.}
    \email{gittenberger@dmg.tuwien.ac.at}
    
    \author{Stephan Wagner}
    \address{Department of Mathematical Sciences, Mathematics Division, Stellenbosch University,
        Private Bag X1, 7602, Matieland, South Africa.}
    \email{swagner@sun.ac.za}

\begin{abstract}
    We study the asymptotic number of certain monotonically labeled increasing trees arising
    from a generalized evolution process.  The main difference between the presented model
    and the classical model of binary increasing trees is that the same label can appear in
    distinct branches of the tree.
    
    In the course of the analysis we develop a method to extract asymptotic information on
    the coefficients of purely formal power series. The method is based on an approximate
    Borel transform (or, more generally, Mittag-Leffler transform) which enables us to quickly
    guess the exponential growth rate. With this guess the sequence is then rescaled and a
    singularity analysis of the generating function of the scaled counting sequence yields
    accurate asymptotics. The actual analysis is based on differential equations and a
    Tauberian argument.
    
    The counting problem for trees of size $n$ exhibits interesting asymptotics involving
    powers of $n$ with irrational exponents.

\end{abstract}

\maketitle


\section{Introduction}

A rooted binary plane tree of size $n$ (meaning that it has $n$ vertices) is called
\emph{monotonically increasingly labeled} with the integers in $\{1, 2,\dots, k\}$ if the vertices
of the tree are labeled with those integers and each sequence of labels along a path from
the root to any leaf is weakly increasing. The concept of a monotonically labeled tree (of fixed
arity $t\ge 2$) has been introduced in the 1980ies by Prodinger and Urbanek~\cite{PU83} and has
then been revisited by Blieberger~\cite{Blieberger87} in the context of Motzkin trees. In the
latter paper, monotonically labeled Motzkin trees are directly related to the enumeration of
expression trees that are built during compilation or in symbolic manipulation systems.

A rooted binary plane tree of size $n$ is called \emph{increasingly labeled}, if it is
monotonically increasingly labeled with the integers in $\{1, 2,\dots, n\}$ and each integer from
$1$ to $n$ appears exactly once. In particular, this implies that the sequences of labels along
the branches are strictly increasing. This model corresponds to the heap data structure in
computer science and is also related to the classical binary search tree model.

In this paper we are interested in a model lying between the two previous ones.  It is in fact
a special subclass of monotonically labeled increasing trees:  \emph{each sequence of labels from
    the root to any leaf is strictly increasing and each integer between $1$ and $k$ must appear in
    the tree}, where $k$ is the largest label. The main difference from the classical model of binary
increasing trees (\emph{cf.} e.g. the book of Drmota~\cite{Drmota09}) is that the same label can
appear in distinct branches of the tree.  Our interest in such a model arises from the following
fact: There is a classical \emph{evolution process}, presented for example in~\cite{Drmota09},
to grow a binary increasing tree by replacing at each step an unlabeled leaf by an internal node
labeled by the step number and attached to two new leaves. Here, we extend the process by
selecting at each step a subset of leaves and replacing each of them by the same structure (the
labeled internal nodes, all with the same integer label, and their two children), thus an
increasing binary tree with repetitions is under construction.

Such a model may serve to describe population evolution processes where each individual can give
birth to two descendants independently of the other individuals. In another paper~\cite{BGN19}
we have presented an increasing model with repetitions of Schröder trees that encodes the
chronology in phylogenetic trees.

Finally, by merging the nodes with the same label we obtain directed acyclic graphs whose nodes
are increasingly labeled (without repetitions). Such an approach introduces a new model of
concurrent processes with synchronization that induces processes whose description is more
expressive than the classical series-parallel model that we have studied
in~\cite{BDGP17,BDGP17bis}.

\medskip
Another main feature of this paper is the methodological aspect. While we are ``only'' presenting
a first order asymptotic analysis of the counting sequence, we introduce an approach to deal with
generating functions which are on the one hand given by some nontrivial functional equation, on
the other hand purely formal power series.

Our approach falls into the guess-and-prove paradigm, which is frequently used in algebraic
combinatorics and partially automatized there. We will apply an approximate Borel transform to
obtain heuristically an equation for the transformed generating functions. Then we use a scaling
indicated by the heuristics in order to deal with a moderately growing sequence. The generating
function of the scaled sequence then admits an asymptotic analysis which leads to the asymptotic
evaluation of the sequence, including a proof of the guessed result after all.

\subsubsection*{Outline of the paper} 

In Section~\ref{basics} we introduce the concept of increasing binary trees with repetitions.
In particular, we develop the evolution process naturally defining such trees and present the
asymptotic behavior of its enumeration sequence.

Section~\ref{binary} is devoted to the asymptotic study of the number of increasing binary trees
with repetitions of size $n$. In the first instance, we partition the problem and get some
recurrence relation. Then we present the methodology based on the approximate Borel transform.
Afterwards, we derive the functional equation from the evolution process and then analyze the
counting sequence, first heuristically and, after having gained the insight from the heuristics,
then exactly.

Then, in Section~\ref{k_ary} we present a brief discussion of the generalization to $k$-ary trees.

\section{Basic concepts and statement of the main result}\label{basics}

The concept of an \emph{increasing tree} is well studied in the literature (\emph{cf.} for example
\cite{Drmota09}).  An increasing tree is defined as a rooted labeled tree where on each path from
the root to a leaf the sequence of labels is increasing.  In fact they are strictly increasing,
since the nodes of a labeled tree with $n$ nodes carry exactly the labels $1,2,\dots,n$. The aim
of the paper is to introduce a weaker model of increasing trees where repetitions of the labels
can appear.

\begin{definition}
    A \emph{weakly increasing binary tree} is
    \begin{itemize}
        \item a binary tree that is not necessarily complete, \emph{i.e.}, the nodes have arity~0,~1 (with two possibilities:
        either a left child or a right one) or 2;
        \item the nodes are labeled according to the following constraints:
        \begin{itemize}
            \item If a node has label $k$, then all integers from 1 to $k-1$ appear as
            labels in the tree. The set of labels is therefore a complete interval of integers 
            of the form $\{1,2,\dots,m\}$ where $m$ is the maximal label occurring in the tree. 
            \item Along each branch, starting from the root, the sequence of labels is (strictly) increasing.
        \end{itemize}
    \end{itemize}
\end{definition}

We can complete an increasing binary tree with repetitions by plugging to each node whose arity
is smaller than $2$ either one or two leaves (without any label) to reach arity 2 for all the
labeled nodes.  We define the size of an increasing binary tree with repetitions as the number
of leaves in the completed binary tree.  This definition of the size will be completely natural
once we will have introduced the way of constructing such trees.

In Figure~\ref{fig:bin-tree} a tree and its associated completed tree are
represented.  Their common size is~$8$.  If we want to expand the tree further, some of the
$\bullet$-leaves will take the label $5$.

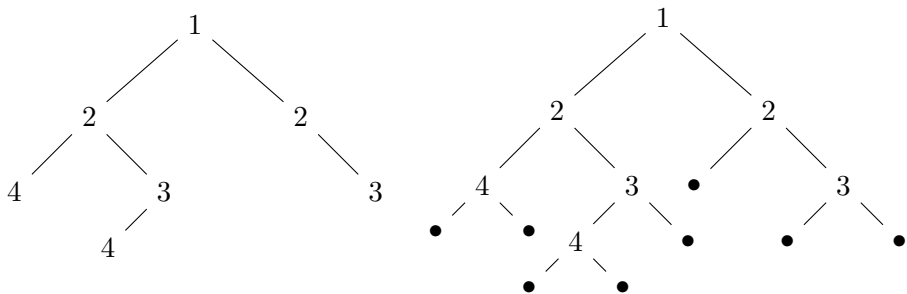
\begin{figure}[htb]
    \begin{center}
        \begin{tabular}{c c c}
            \hspace*{-5mm}
            \begin{tikzpicture}[node distance=35pt]
            \node(a) {$1$};
            \node[below of=a] (aa) {};
            \node[left of=aa, node distance=40pt] (b) {$2$};
            \draw (a) -- (b);
            \node[below right of=b, node distance=40pt] (c) {$3$};
            \draw (b) -- (c);
            \node[below left of=c, node distance=30pt] (cc) {$4$};
            \draw (c) -- (cc);
            \node[below left of=cc, node distance=25pt] (ee) {};
            \node[below left of=b, node distance=40pt] (e) {$4$};
            \draw (b) -- (e);		
            \node[right of=aa, node distance=40pt] (d) {$2$};
            \draw (a) -- (d);
            \node[below right of=d, node distance=40pt] (ddd) {$3$};
            \draw (d) -- (ddd);
            \end{tikzpicture}
            & 
            \begin{tikzpicture}[node distance=35pt]
            \node(a) {$1$};
            \node[below of=a] (aa) {};
            \node[left of=aa, node distance=40pt] (b) {$2$};
            \draw (a) -- (b);
            \node[below right of=b, node distance=40pt] (c) {$3$};
            \draw (b) -- (c);
            \node[below left of=c, node distance=30pt] (cc) {$4$};
            \draw (c) -- (cc);
            \node[below left of=cc, node distance=25pt] (ee) {$\bullet$};
            \draw (cc) -- (ee);
            \node[below right of=cc, node distance=25pt] (eee) {$\bullet$};
            \draw (cc) -- (eee);			
            \node[below right of=c, node distance=30pt] (ccc) {$\bullet$};
            \draw (c) -- (ccc);			
            \node[below left of=b, node distance=40pt] (e) {$4$};
            \draw (b) -- (e);		
            \node[below left of=e, node distance=25pt] (ee) {$\bullet$};
            \draw (e) -- (ee);
            \node[below right of=e, node distance=25pt] (eee) {$\bullet$};
            \draw (e) -- (eee);				
            \node[right of=aa, node distance=40pt] (d) {$2$};
            \draw (a) -- (d);
            \node[below left of=d, node distance=40pt] (dd) {$\bullet$};
            \draw (d) -- (dd);
            \node[below right of=d, node distance=40pt] (ddd) {$3$};
            \draw (d) -- (ddd);
            \node[below left of=ddd, node distance=30pt] (ee) {$\bullet$};
            \draw (ddd) -- (ee);
            \node[below right of=ddd, node distance=30pt] (eee) {$\bullet$};
            \draw (ddd) -- (eee);
            \end{tikzpicture}
        \end{tabular}
    \end{center}
    \caption{Left: a weakly increasing tree with $7$ nodes and $4$ distinct labels;
        Right: its associated completed tree. 
        \label{fig:bin-tree}}
\end{figure}

Let us introduce a \emph{combinatorial evolution process} to build weakly increasing trees.
In order to construct a weakly increasing tree whose greatest label is $m$, 
start with the weakly increasing tree of size 2 (a root with label $1$)
and repeat the following step $(m-1)$ times. At step $i \in \{1, \dots, m-1\}$,
choose a non-empty subset of $\bullet$-leaves from the current completed weakly increasing tree
and replace each of them by a tree with root labeled by $(i+1)$ and two $\bullet$-leaves.
At the end of the process remove the $\bullet$-leaves to formally obtain a weakly increasing tree.

\begin{lemma}\label{lem:single_process}
    Given a weakly increasing tree, there is a single combinatorial evolution process
    that builds it.
\end{lemma}

The combinatorial structure of weakly increasing trees being now formally defined, we denote by
$B_n$ the number of weakly increasing trees of size $n$.  We will prove the following quantitative
result.

\begin{theorem}\label{theo:nb_bin_tree}
    The number of weakly increasing binary trees of size $n$ is
    asymptotically given by 
    \[
    B_n \simn \eta \; n^{-\ln 2} \(\rcp{\ln 2}\)^n  (n-1)!
    \]
    where $\eta \approx 0.647852\dots$.
\end{theorem}

This result can be compared to the number of classical increasing binary trees (without label
repetition) with $(n-1)$ labeled nodes, which equals $(n-1)!$. The number $B_n$ is
exponentially greater than the latter one.  For the classical increasing model, the reader can
refer to Flajolet and Sedgewick's book~\cite[p. 143]{FS09}.

\section{Enumeration of weakly increasing binary trees}\label{binary}

Using the combinatorial evolution process to build trees (described in the previous section) and
its associated Lemma~\ref{lem:single_process}, we get directly a recurrence for the partition of
the set of all weakly increasing trees of size $n$ according to their maximal label $m$:

\begin{align}
B_{1,2} & = 1 \nonumber \\
B_{1,n} & = 0 \qquad \text{if }n\neq 2 \nonumber \\
B_{m,n} & = \sum_{\ell = 1}^{n-m+2} \binom{n-\ell}{\ell} B_{m-1, n-\ell},     \label{B_r_rec}
\end{align}
where $B_{m,n}$ is the number of weakly increasing trees with $n$ nodes in which exactly $m$
distinct labels occur.  We remark that $B_n=\sum_{m\ge 1} B_{m,n}$, and thus the first terms of
$(B_n)_{n\ge 0}$ are

{
    \small
    \[0,0,1,2,7,34,214,1652,15121,160110,1925442,25924260,386354366,6314171932,\dots\]
}

The first terms of our sequence coincide with those of a shifted version of the sequence
\texttt{A171792} in OEIS\footnote{OEIS means \emph{Online Encyclopedia of Integer Sequences} that
    can be reached at \\\url{http://oeis.org/}}.  Some properties of this sequence are stated there,
but no combinatorial meaning is given.

\subsection{Methodology of the approach}\label{meth}

We start with deriving a functional equation for the generating function.  As we are not dealing with
labeled structures in the sense of \cite[Chapter~II]{FS09} (where no repetitions are allowed),
we cannot apply the symbolic method for labeled structures to obtain a functional equation.
Instead, we use the evolution process mentioned in the introduction and the symbolic method for
unlabeled structures, which yields a functional equation for the ordinary generating function
after all. As this function is a purely formal series, no analytic methods apply.

A Borel transform makes the power series become analytic in some vicinity of the origin, as it
transforms a power series $\sum_n a_n z^n$ into $\sum_n a_n z^n/n!$. As we
cannot calculate the Borel transform of the functional equation we obtained, we will use an
approximate Borel transform: First, we transform the functional equation into a recurrence
relation. Then we guess that the Borel transform of the generating function has a positive radius
of convergence and replace $B_n$ by the approximation $n!\alpha^n$, from which we derive a simple
functional for the exponential generating function. This function does not appear as a closed
form expression on both sides of the equation. But as the terms appearing in the summation contain
some rapidly growing sequences, we guess that only a few terms are actually asymptotically
relevant. Then we simplify and extend the range of summation. Actually, we prune a major part of
the sum, then simplify the remaining terms, and then replace the pruned summands by a sum over
simplified terms. This resembles the idea of the saddle point method for complex contour integrals
(\emph{cf.}~\cite[Sec.~VIII.3]{FS09}), where the integrand is approximated locally and the tails
of the integral are exchanged.

Now we are faced with a rather simple equation from which we can read off the dominant
singularity, \emph{i.e.}, the singularity which lies closest to the origin and which reveals the
exponential growth rate of the coefficients of the generating function. Thus we guess the
exponential growth rate of $B_n/n!$, more precisely, we guess that $\ln B_n \sim \ln( n!
\alpha^n)$ with the value $\alpha$ we just obtained.

This guess now indicates that the suitably scaled sequence has only polynomial growth and hence
its generating function is amenable to a singularity analysis. Indeed, we will prove an \emph{a
    priori} bound for the scaled sequence and derive then a first-order differential equation for its
generating function. Finally, a singularity analysis and a Tauberian theorem yield the asymptotic
equivalent of the scaled sequence, and thus for $B_n$ as well. This includes a proof of the
guessed exponential growth rate as well. 

\brem 
We remark that the same methodology applies when a Mittag-Leffler transform is used to make the
power series analytic. The Mittag-Leffler transform was introduced by Mittag-Leffler~\cite{ML1903}
and is a generalization of the Borel transform,
which transforms a power series $\sum_n a_n z^n$ into $\sum_n a_n z^n/\Gamma(1+\alpha n)$ for some
specific $\alpha$. This roughly corresponds to replacing the $n!$ in the Borel transformed series
by $n!^\alpha$. Further information and more recent developments concerning the Mittag-Leffler
transform can be found in \cite{Za96} 
\erem

\subsection{The generating function of the counting sequence}

Let us now introduce the \emph{ordinary generating series} $B(z)$ associated to the sequence $B_n$:
\[
B(z) = \sum_{n\geq 0} B_n z^n.
\]
The variable $z$ marks the $\bullet$-leaves in the weakly increasing trees.  Although the trees
are labeled, we use an ordinary generating series, because in this context the evolution process
directly turns into a combinatorial specification and hence a functional equation satisfied by
the series $B(z)$.  Recall that at each step some $\bullet$-leaves are replaced by a deterministic
labeled node with two $\bullet$-leaves; thus we get
\begin{equation} \label{prev_funeq}
B(z) = z^2 + B(z + z^2) - B(z).
\end{equation} 
In fact, a tree is either the smallest tree (the root labeled by $1$ with two $\bullet$-leaves)
or it is obtained after the expansion of a tree where some $\bullet$-leaves do not change, and
the other ones are replaced by a labeled node and two $\bullet$-leaves: $z\rightarrow z^2$).
But at each step at least one leaf must be chosen, thus we must remove the trees where no leaf
has been chosen; these have generating function $B(z)$. The functional 
equation~\eqref{prev_funeq} can be rewritten as
\begin{equation}\label{eq:function_bin_trees}
B(z) = \frac{1}{2} \left( z^2 + B\left(z + z^2\right) \right).
\end{equation}
Computing the coefficients from Equation~\eqref{eq:function_bin_trees}, we prove that our
sequence $\left(B_n\right)_{n\ge 0}$ is indeed a shifted version of OEIS \texttt{A171792}.

It is remarkable that the most natural description here uses ordinary generating functions, unlike
the exponential ones that are used for classical increasing trees. As mentioned in
Section~\ref{meth}, the previous section, the reason is that due to the label repetitions we
cannot specify this class directly, but only using the evolution process, which puts us into the
world of ordinary generating functions. We may apply the combinatorial Borel transform on
equation~\eqref{eq:function_bin_trees}, which translates $B(z)$ into its exponential counterpart.
But this does neither reveal another natural combinatorial way for defining weakly increasing
trees nor turn the functional equation~\eqref{eq:function_bin_trees} into a simple one for the
exponential generating function. For deriving the asymptotics, we will now appeal to the approach
outlined in Section~\ref{meth}.

\subsection{Analysis of the functional equation -- heuristics} 

Now we turn to the actual enumeration problem which amounts to the analysis of the function given
in~\eqref{eq:function_bin_trees}. First, we read off coefficients in~\eqref{eq:function_bin_trees}
and get a recurrence relation for $B_n=[z^n] B(z)$, which is, of course, in compliance with
\eqref{B_r_rec}: starting with $B_2=1$, we then get
\begin{align} \label{B_n_rec}
B_n&=\sum_{\ell=1}^{\lfloor\frac n2\rfloor} \binom{n-\ell}{\ell} B_{n-\ell} \\
&=\sum_{p=n-\lfloor\frac n2\rfloor}^{n-1}\binom{p}{n-p} B_p. \label{B_p}
\end{align} 

First, let us introduce a combinatorial interpretation of Equation~\eqref{B_n_rec}, which enables
us to directly deduce \eqref{B_p} from \eqref{B_n_rec}. In the first recurrence, we state that
a tree with $n$ leaves is obtained by extending a tree with $n-\ell$ leaves in which we choose
$\ell$ leaves, each one being then replaced by an internal node (with a deterministic label
induced by the step number in the construction) to which two leaves are attached.

Looking at this recurrence, we immediately observe that $B_n\ge (n-1)!$, thus $B(z)$ is only a
formal power series. To get a first guess of the asymptotic behavior of $B_n$, we start with the
following heuristic (first step of the approximate Borel transform): 
\emph{Assume that $B_n \simn \alpha^n \; n!$, for some $\alpha>0$.}

Then the asymptotic analysis of $B_n$ could be done by a singularity analysis 
(see in particular~\cite{FO90,FS09}) of the exponential generating function
\[
\hat B(z)=\sum_{n\ge 0} B_n \frac{z^n}{n!}.
\]
Set $\phi_n:=\alpha^n n!$ and add $B_n$ to both sides of equation~\eqref{B_p}. This gives,
when summing up over all $n$ and assuming $B_n=\phi_n$, based on the left-hand side of~\eqref{B_p}
\[
2\sum_{n\ge 0} B_n \frac{z^n}{n!} = 2\sum_{n\ge 0} \phi_n \frac{z^n}{n!}  
= \frac{2}{1-\alpha z}.
\]
By using the right-hand side of~\eqref{B_p} we deduce
\begin{align*} 
\frac{2}{1-\alpha z} &= \sum_{n\ge 0} \frac{z^n}{n!}
\sum_{p=n-\lfloor\frac n2\rfloor}^{n} \binom{p}{n-p} \phi_p \\
&=\sum_{n\ge 0} z^n \sum_{p=n-\lfloor\frac n2\rfloor}^{n} \frac{\phi_p}{p!} \cdot \frac1{(n-p)!}
\cdot\frac{p!^2}{n!(2p-n)!}.
\end{align*} 
Note that $\frac{p!^2}{n!(2p-n)!} \approx 1$ for $p\approx n$. 
If $p$ is getting smaller then $\frac{p!^2}{n!(2p-n)!}$ rapidly tends to~$0$. And so does
$\frac1{(n-p)!}$. Thus, only the last few terms of the inner sum should already almost give its
value. Thus, let us assume that 
\[
\frac{2}{1-\alpha z} \sim \sum_{n\ge 0} z^n \sum_{p=0}^{n} \frac{\phi_p}{p!} \cdot
\frac1{(n-p)!} =\frac{e^z}{1-\alpha z}.
\]
But now, we see that both generating functions have a unique dominant singularity at $1/\alpha$
and they are approximately the same function. Thus, as $z\to 1/\alpha$, we must have that $2\sim
e^{1/\alpha}$, which yields $\alpha=1/\ln 2\approx 1.442695041\dots$. 

Note, that the reasoning above is only heuristic. There are many inaccuracies in our
arguments, so we have not proved anything so far. However, comparing $(\ln 2)^{-n} \; n!$ with
the first $1000$ values of $(B_n)$ indicates that $B_n\sim b_n (\ln 2)^{-n} (n-1)! $ where $b_n\to 0$ at
a slower rate than $1/n$. In Figure~\ref{fig:asympt_approx} the normalized values $B_n / ( (\ln 2)^{-n} (n-1)!  )$
are represented by the blue curve. The green one represents the function $n \mapsto 1 / \sqrt{n}$
and the red one is for the function $n \mapsto 1 / n$. The representations are given for $n=25\dots 1000$.
\begin{figure}[htb]
    \begin{center}
        \includegraphics[width=0.7\textwidth]{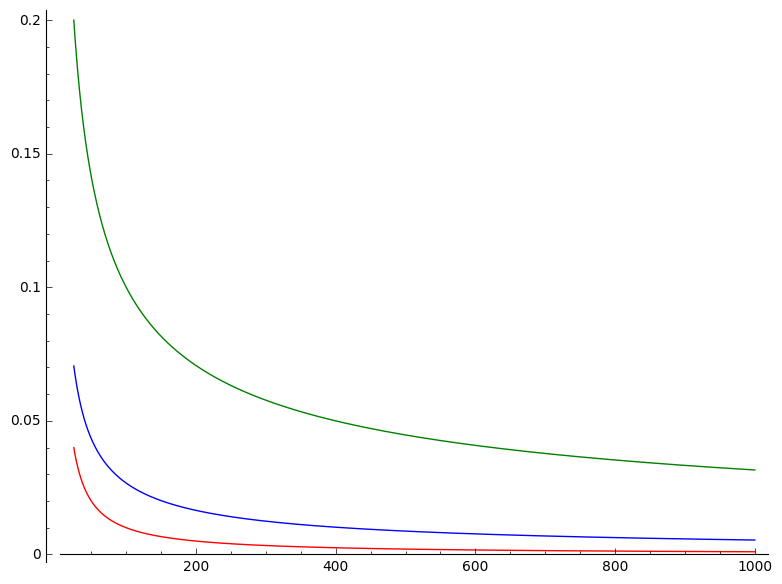}
    \end{center}
    \caption{The blue curve corresponds to the function $n \mapsto B_n / ( (\ln 2)^{-n} (n-1)!  )$.
        \label{fig:asympt_approx}}
\end{figure}

\subsection{Analysis of the functional equation -- asymptotics}

With the heuristic observation of the last section in mind, we scale the original counting
sequence and define a new sequence $(b_n)_{n\ge 2}$ by
\[
b_n:= \frac{B_n}{(\ln 2)^{-n}(n-1)!} 
\]
and set $b_0=b_1=0$. If we set $\alpha=1/\ln2$, then the recurrence~\eqref{B_n_rec} becomes 
\begin{equation} \label{bnrec}
b_n = \sum_{\ell=1}^{\lfloor\frac n2\rfloor}  \frac{\alpha^{-\ell}}{\ell!} \cdot
\frac{(n-\ell)(n-\ell-1)\cdots (n-2\ell+1)}{(n-1)(n-2)\cdots (n-\ell)} \cdot b_{n-\ell},
\end{equation} 
where the first values are $b_0=b_1=0$ and $b_2 = (\ln 2)^2$, of course. 
To proceed, we need to analyze the second factor. 
To simplify notations set 
\begin{equation} \label{gammas}
\gamma_{n,\ell} =\frac{(n-\ell)(n-\ell-1)\cdots (n-2\ell+1)}{(n-1)(n-2)\cdots (n-\ell)}.
\end{equation} 
First let us establish tight bounds for these numbers.
\bl \label{gamma_lemma}
For $1 \le \ell < n$ we have
\[
1-\frac{\ell(\ell-1)}n-\frac{\ell(\ell-1)^2}{n^2} \le \gamma_{n,\ell}\le  
1-\frac{\ell(\ell-1)}n+\frac{\ell(\ell-1)^3}{2n^2}.
\]
\el

\bpf
To show the upper bound, we write 
\begin{align} 
\gamma_{n,\ell}&=\(1-\frac{\ell}{n-1}\)\(1-\frac{\ell-1}{n-2}\)\cdots
\(1-\frac{\ell}{n-\ell}\) \nonumber \\ 
&\le \(1-\frac{\ell-1}{n}\)^{\ell}. \label{gamma_aux2}
\end{align} 
Using the fact that $(1-x)^r\le 1-rx+\binom{r}{2} x^2$ for $0<x<1$ and $r\in\mathbb{N}$,
we get the stated upper bound.

Now we turn to the lower bound. 
Let $(H_n)_{n\geq1}$ denote the sequence of harmonic numbers.
If we cancel the factor $n-\ell$ in \eqref{gammas}, we can write 
\[
\gamma_{n,\ell}=\(1-\frac{\ell}{n-1}\)\(1-\frac{\ell}{n-2}\)\cdots \(1-\frac{\ell}{n-\ell+1}\)
\]
which implies 
\begin{align} 
\gamma_{n,\ell}&\ge 1-\ell \(\rcp{n-1}+\rcp{n-2}+\cdots+\rcp{n-\ell+1}\) 
=1-\ell \(H_{n-1}-H_{n-\ell} \). \label{gamma__aux} 
\end{align} 
Recall (\emph{cf.} for example~\cite{TT71}) that the sequence $(H_n - \ln n )_n$
is monotonically decreasing, thus we deduce
\begin{align*} 
\gamma_{n,\ell} &\ge 1-\ell \(\ln (n-1)-\ln(n-\ell)\)=1-\ell\ln\(\frac{n-1}{n-\ell}\) 
= 1-\ell\ln\(\rcp{1-\frac{\ell-1}{n-1}}\)\\
& \ge 1-\ell\ln\(\rcp{1-\frac{\ell-1}{n}}\).
\end{align*} 
Since $\gamma_{n,\ell}$ is defined for $\ell\le\lfloor
n/2\rfloor$, the result now follows from the inequality $\ln\rcp{1-x}\le x+x^2$, which holds for
$0\le x\le 1/2$. 
\epf

Next, we prove a first bound for $b_n$. 



\bl\label{bnbound}
The sequence $(b_n)_{n\ge 0}$ defined by $b_0=b_1=0$, $b_2=(\ln 2)^2$, and 
Equation~\eqref{bnrec} for $n\ge 3$ satisfies for all $n\in \mathbb{N}$ 
\[
0 \le b_n\le \frac{1}{n^{\eps}}
\]
for some sufficiently small $\eps>0$.
\el

\bpf
We use induction on $n$. Apparently, the claimed bound holds for $b_0,b_1,b_2$ if $\eps$ is small
enough. Now assume that $b_k\le \frac{1}{k^{\eps}}$ for $k=1,\dots,n-1$. Using the
recurrence~\eqref{bnrec} together with the bound~\eqref{gamma_aux2} we obtain
\begin{align}
b_n &\le \sum_{\ell=1}^{\lfloor n/2 \rfloor} \frac{\alpha^{-\ell}}{\ell!} 
\cdot \left(1-\frac{\ell-1}{n}\right)^{\ell} \cdot \rcp{(n-\ell)^{\eps}} \nonumber\\
&= \frac1{n^{\eps}}\( \rcp{\alpha}\(1-\rcp n\)^{-\eps} 
+ \sum_{\ell=2}^{\lfloor n/2 \rfloor} \frac{\alpha^{-\ell}}{\ell!} \label{bn_first}
\left(1-\frac{\ell-1}n\right)^{\ell} \(1-\frac{\ell}n\)^{-\eps}\). 
\end{align}
Now observe that for $0\le x\le 1/2$ and small $\eps$ (e.g. $\eps=1/100$) we have
$(1-x)^{-\eps}<1+2\eps x$. Thus we get the estimates 
\[
\(1-\rcp n\)^{-\eps}<1+\frac{2\eps}n,\qquad 
\(1-\frac{\ell}n\)^{-\eps}<1+\frac{2\ell\eps}n \le \(1+\frac{2\eps}n\)^\ell
\]
and in the sum in \eqref{bn_first} we have moreover $\left(1-\frac{\ell-1}n\right)^{\ell}\le
\(1-\rcp n\)^\ell$. Using these bounds and extending the range of summation to infinity, we obtain 
\begin{align*}
b_n \le \frac1{n^{\eps}} \Bigg( & \rcp{\alpha}\(1+\frac{2\eps}n\)+\exp\(\rcp\alpha\(1-\rcp
n\)\(1+\frac{2\eps}n\)\) \\
&  -1-\rcp\alpha\(1-\rcp n\)\(1+\frac{2\eps}n\)\Bigg).
\end{align*}
The exponential function satisfies  
\[
e^x-1-x\le cx^2,\quad\text{ with }\quad c=\frac{1-\ln 2}{\ln(2)^2} \quad \text{ and }\quad 0\le x\le \rcp
\alpha=\ln 2.
\]
Hence we infer (notice that $c=\alpha^2\(1-\rcp\alpha\)$)
\begin{align*} 
b_n &\le \frac1{n^{\eps}} \(\rcp{\alpha}\(1+\frac{2\eps}n\)+\frac{c}{\alpha^2}
\(1-\rcp n\)^2\(1+\frac{2\eps}n\)^2\) \\
&=\frac1{n^{\eps}} \(\rcp{\alpha}\(1+\frac{2\eps}n-\(1+\frac{2\eps}n\)^2\(1-\rcp
n\)^2\)+\(1+\frac{2\eps}n\)^2\(1-\rcp n\)^2\)\\
&\le \frac1{n^{\eps}} \(\rcp{\alpha}\cdot\frac2n\(1+\frac{2\eps}n\)
+\(1+\frac{2\eps}n\)^2\(1-\rcp n\)^2\). 
\end{align*} 
The last factor is an increasing function in $n$. Therefore, replacing it by its limit gives
another upper bound. As this limit is equal to 1, the proof is complete. 
\epf

We are now ready to analyze precisely the asymptotic behavior of $b_n$.
Starting with Equation~\eqref{bnrec}, let us define the correction sequence $(a_n)$ to be
\begin{align*}
a_n =  &\sum_{\ell=1}^{\lfloor\frac n2\rfloor}  \frac{\alpha^{-\ell}}{\ell!} \cdot
\left( \gamma_{n,\ell}
- 1 + \frac{\ell(\ell-1)}{n}\right) \cdot b_{n-\ell} \\
& - \sum_{\ell=\lfloor\frac n2\rfloor+1}^{n} \frac{\alpha^{-\ell}}{\ell!} \cdot
\left( 1 - \frac{\ell(\ell-1)}{n}\right) \cdot b_{n-\ell}.
\end{align*}
We obviously get
\begin{equation} \label{eq:an}
b_n = a_n + \sum_{\ell=1}^n \frac{\alpha^{-\ell}}{\ell!}
\left( 1 - \frac{\ell (\ell-1)}{n} \right) b_{n-\ell}.
\end{equation}
We now associate the generating function $a(z) = \sum_{n\ge0} a_n z^n$.
\begin{lemma}\label{bound_an}
    We have $a_n = O(n^{-2-\eps})$, as $n\to\infty$ and for some $\eps>0$. 
    Thus the functions $a(z)$ and $a'(z)$ are bounded for $z\rightarrow 1^-$.
\end{lemma}
\begin{proof}
    Lemma~\ref{gamma_lemma} gives 
    \[
    \gamma_{n,\ell} - 1 + \frac{\ell(\ell-1)}{n} = O\left( \frac{\ell^4}{n^2} \right).
    \]
    Now using Lemma~\ref{bnbound} completes the proof. 
\end{proof}
By adding the coefficient $b_n$ to each part of Equation~\eqref{eq:an}
and then  multiplying by $n$, we get
\begin{equation}\label{rec:bn}
2n b_n = n a_n + n \sum_{\ell=0}^n \frac{\alpha^{-\ell}}{\ell!} b_{n-\ell}
- \sum_{\ell=0}^n \frac{\alpha^{-\ell}}{\ell!} \ell (\ell-1) b_{n-\ell}.
\end{equation}
Translating this recurrence into the generating series context we obtain
the following differential equation.
\begin{lemma}\label{equadiff_sol}
    The generating function $b(z) = \sum_n b_n z^n$ satisfies
    \[
    \left( 2 - e^{z/\alpha} \right) b'(z) 
    + \left( \frac{z}{\alpha^2} - \frac{1}{\alpha} \right)e^{z/\alpha} b(z)
    = a'(z).
    \]
\end{lemma}
\begin{proof}
    First we observe that 
    \[
    \sum_{\ell\ge 0} \ell(\ell-1)\alpha^{-\ell}
    \frac{z^\ell}{\ell!}=\frac{z^2}{\alpha^2}e^{z/\alpha},
    \]
    thus we deduce from Equation~\eqref{rec:bn}
    the following equation in the context of generating functions:
    \[
    2z b'(z) = z a'(z) + z \left( e^{z/\alpha} b(z) \right)'
    - \left(\frac{z}{\alpha}\right)^2 e^{z/\alpha} b(z).
    \]
    which implies the assertion.
\end{proof}

We are now ready to study the behavior of $b(z)$ around its singularity $1$.
\begin{lemma}\label{lem:asympt}
    The generating function $b(z)$ satisfies
    \begin{align*}
    b(z) \underset{z\rightarrow 1^-}\sim \beta  \left( 1-z \right)^{-1 + 1/\alpha }\qquad
    \text{where } \beta = \frac{e^{\pi^2/12}\; \alpha^{1-1/\alpha}}{2^{1-1/(2\alpha)}} \;
    \int_0^1 \frac{f(t) a'(t)}{2- e^{t/\alpha}} \mathrm{d}t.
    \end{align*}
\end{lemma}
\begin{proof}
    The generic solution of the homogeneous differential equation
    \[
    \left( 2 - e^{z/\alpha} \right) y'(z)
    + \left( \frac{z}{\alpha^2} - \frac{1}{\alpha} \right)e^{z/\alpha} y(z)
    = 0
    \]
    is $y(z)=Cg(z)$ with 
    \[
    g(z)= \exp \left( -\int_{0}^{z} \left( \frac{t}{\alpha^2}
    - \frac{1}{\alpha} \right) \frac{e^{t/\alpha}}{ 2 - e^{t/\alpha}} \mathrm{d}t \right).
    \]
    By variation of constants we obtain 
    $C'(z)\cdot \left( 2 - e^{z/\alpha} \right) g(z)= a'(z)$
    and hence, as $b_0=0$, 
    \begin{equation} \label{h_int}
    b(z) = g(z) \int_0^z \frac{a'(t)}{\left(2- e^{t/\alpha}\right)g(t)} \mathrm{d}t.
    \end{equation} 
    
    Observe that $2 - e^{z/\alpha} \sim 2(1-z)/\alpha$, as $z\to 1$. This allows us to expand
    $g(z)$ asymptotically, which gives 
    \[
    g(z) \underset{z\rightarrow 1^-}\sim e^{\pi^2/12} \;
    \frac{\alpha^{1-1/\alpha}}{2^{1-1/(2\alpha)}} \;
    \left( 1 - z \right)^{-1+1/\alpha}. 
    \]
    Expanding $2-e^{z/\alpha}$ around $z=1$ 
    we obtain $1/(2-e^{z/\alpha}) \sim \alpha/2 \; (1-z)^{-1}$, as $z\rightarrow 1^-$. 
    So, from Lemma~\ref{bound_an} we deduce
    \[
    \frac{a'(z)}{\left(2- e^{z/\alpha}\right)g(z)} \underset{z\rightarrow 1^-}\sim
    \frac{ \alpha \; e^{-\pi^2/12} \; a'(1)}{2} \left( 1-z \right)^{-1/\alpha},
    \]
    which guarantees that the integral in~\eqref{h_int} is bounded for $0\le z\le 1$ and finally proves 
    the claimed result.
\end{proof}

\begin{proof}[Proof of Theorem~\ref{theo:nb_bin_tree}]
    Finally, using the result of Lemma~\ref{lem:asympt}
    and applying a standard Tauberian theorem, recalled in~\cite[Theorem VI.13]{FS09}, we deduce
    that 
    \[
    b_n \underset{n\rightarrow \infty}\sim \frac{\beta}{\Gamma(1-1/\alpha)} n^{-1/\alpha}.
    \]
    Replacing $\alpha$ by its value $1/\ln 2$ and defining the constant
    \[
    \eta = \frac{e^{\pi^2/12} \; \alpha^{1-1/\alpha}}{2^{1-1/(2\alpha)} \; \Gamma(1-1/\alpha)} \;
    \int_0^1 \frac{f(t) a'(t)}{2- e^{t/\alpha}} \mathrm{d}t,
    \]
    implies the main result after all.
\end{proof}

\section{Higher arity weakly increasing trees}\label{k_ary}

In this section we briefly discuss how our results generalize to $k$-ary weakly increasing trees
with repetitions. 
The definitions about the binary case can be adapted in an obvious way. 
The size of the structures corresponds to the number of leaves of the completed $k$-ary tree
and the generating function
$G(z)=\sum_{n\ge k} G_n z^n$ where $G_n$ is the number of $k$-ary weakly increasing trees
(with repetitions) of size $n$. In a similar
way as in the binary case we obtain the functional equation 
\[
G(z) = \frac{1}{2} \left(z^k + G(z+z^k)\right).
\]
We follow the same strategy as in the binary case to study the asymptotic behavior of $G_n$.
Recall the recurrence for the binary case:
$B_n=\sum_{\ell=1}^{\lfloor\frac n2\rfloor} \binom{n-\ell}{\ell} B_{n-\ell}$. It was obtained
through the following expansion
\begin{align*} 
B_n &=\sum_{\ell=1}^{\lfloor\frac n2\rfloor} B_{n-\ell} [z^n](z+z^2)^{n-\ell} 
=\sum_{\ell=1}^{\lfloor\frac n2\rfloor} B_{n-\ell} [z^{\ell}](1+z)^{n-\ell}. 
\end{align*}  
Thus we must replace the term $z^2$ on the right-hand side of the first equation by $z^k$. 
From this, we get the recurrence for the $k$-ary case: 
\begin{align} 
G_n&=\sum_{\ell=1}^{\lfloor n-\frac nk\rfloor} G_{n-\ell} [z^n](z+z^k)^{n-\ell} 
=\sum_{\ell=1}^{\lfloor n-\frac nk\rfloor} G_{n-\ell} [z^{\ell}](1+z^{k-1})^{n-\ell} \nonumber \\
&=\sum_{\begin{array}{c} \Sc \ell=1, \\ \Sc\ell\equiv 0\mod k-1\end{array}}^{\lfloor n-\frac nk\rfloor}
\binom{n-\ell}{\frac{\ell}{k-1}} G_{n-\ell} 
=\sum_{s=1}^{\lfloor \frac nk\rfloor} \binom{n-(k-1)s}s  G_{n-(k-1)s} \label{Bnlargek} \\
&=\sum_{\begin{array}{c} \Sc p=n-(k-1)\lfloor\frac nk\rfloor, \\ \Sc p\equiv n \mod k-1\end{array}}^{n-k+1}
\binom{p}{\frac{n-p}{k-1}} G_p. \nonumber
\end{align}
The expressions in the last two lines show that only particular terms of the sequence $(G_n)_{n\ge 0}$
are nonzero. This is not a surprise, since the arity constraint implies that $G_n\neq 0$ if and only if
$n\equiv 1\mod k-1$. Thus we set $H_n=G_{1+n(k-1)}$. From Equation~\eqref{Bnlargek} we obtain 
\begin{equation} \label{Crec}
H_n = \sum_{s=1}^{\lfloor n-\frac{n-1}k\rfloor} \binom{1+(n-s)(k-1)}s H_{n-s} = 
\sum_{s=\lceil \frac{n-1}k\rceil}^{n} \binom{1+s(k-1)}{n-s} H_s. 
\end{equation} 

Let us define $h_n$ as $H_n = h_n \; (k-1)^n \; (\ln 2)^{-n} \; n!$.
Then, $h_1 = \ln 2 / (k-1)$ and by~\eqref{Crec} we have 
\begin{equation} \label{hnrec}
h_n=\sum_{s=1}^{\lfloor n-\frac{n-1}k\rfloor}  \left( \frac{\ln 2}{k-1} \right)^s \frac{1}{s!} \; \delta_{n,s} \; h_{n-s},
\end{equation} 
where
\begin{equation}\label{delta}
\delta_{n,s} = \frac{(1+(n-s)(k-1))! \ (n-s)!}{(1+(n-s)(k-1)-s)! \; n!}.
\end{equation}

\bl
\label{delta_lemma}
Let $\delta_{n,s}$ be defined by \eqref{delta}. Then we have for $n>0$
\[
\delta_{n,1} = (k-1) \left( 1 - \frac{1}{n} + \frac{1}{n \; (k-1)} \right),  
\]
and for $1 < s \le \lfloor n - \frac{n-1}{k} \rfloor$,
\[
0 \le \delta_{n,s} \le (k-1)^s \left( 1 - \frac{s}{n} \right).
\]
\el
\bpf 
A direct calculation for $\delta_{n,1}$ gives the result.
Let us now prove by induction on $n$ that for all $s\in \{2, \dots, \lfloor n - (n-1)/k\rfloor\}$
we have $\delta_{n,s} \le (k-1)^s ( 1 - s/n )$. 
When $s=2$ (thus $n\ge 3$) we have an extremal case:
\begin{align*}
\delta_{n,2} &= \frac{((n-2)(k-1)+1)(n-2)(k-1)}{n(n-1)} \\
&= (k-1)^2\left( 1-\frac{2}{n} + \frac{1}{n(k-1)} \right) \left( 1 - \frac{1}{n-1}\right)\\
&= (k-1)^2\left( 1-\frac{2}{n} + \frac{(n-2)(2-k)}{n(n-1)(k-1)} \right).
\end{align*}
Since the last fraction is negative because $k\ge 3$, we obtain
\[
\delta_{n,2} \le (k-1)^2\left( 1-\frac{2}{n}\right).
\]
So the property is true when $n=3$, and $s=2$.
Let us suppose the property is true for $n$ and all
$2\leq s \leq \lfloor n - \frac{n-1}{k}\rfloor$.

Let $s \in \{3, \dots, \lfloor n+1 - \frac{n}{k}\rfloor\}$.
\[
\delta_{n+1,s} = \delta_{n,s-1} \frac{(n+1-s)(k-1)+2-s}{n+1}.
\]
Since $s\leq \lfloor n+1 - \frac{n}{k}\rfloor$, then $s-1 \leq  \lfloor n - \frac{n-1}{k}\rfloor$.
Thus we can use the property for $\delta_{n, s-1}$ and
\begin{align*}
\delta_{n+1,s} & \leq (k-1)^{s-1} \left( 1 - \frac{s-1}{n} \right) \frac{(n+1-s)(k-1)+2-s}{n+1} \\
& \leq (k-1)^{s} \left( 1 - \frac{s-1}{n} \right) \left( 1- \frac{s}{n+1} - \frac{s-2}{(n+1)(k-1)} \right)\\
& \leq (k-1)^{s} \left( 1- \frac{s}{n+1} \right).
\end{align*}
Thus the stated result is proved.
\epf

\begin{corollary}
    The sequence $(h_n)_{n\ge 0}$ defined by $h_0=0$, $h_1=\ln 2 / (k-1)$, and 
    Equation~\eqref{hnrec} for $n>1$, satisfies for all $n\in \mathbb{N}$ 
    \[
    0 \le h_n\le \frac{1}{n^{\ln 2}}.
    \]
\end{corollary}
\begin{proof}
    Let us prove the result by induction. Since $k\ge3$, the result is true for $h_1$.
    Suppose the result is correct until index $n-1$.
    We are now interested in $h_n$. Using Lemma~\ref{delta_lemma} yields
    \begin{align*}
    h_n &= \frac{\ln 2}{k-1} \delta_{n,1} h_ {n-1} + 
    \sum_{s=2}^{\lfloor n - \frac{n-1}{k-1}\rfloor} \frac{(\ln 2)^s}{(k-1)^s s!} 
    \delta_{n,s} h_{n-s}\\
    & \leq \frac{\ln 2}{(n-1)^{\ln 2}} \left( 1 - \frac{1}{n} + \frac{1}{n(k-1)} \right) + 
    \frac{1}{n^{\ln 2}} \sum_{s=2}^{\lfloor n+1 - \frac{n}{k-1}\rfloor}
    \frac{(\ln 2)^s}{s!}\left(1-\frac{s}{n}\right)^{1-\ln 2}\\
    & \leq \frac{1}{n^{\ln 2}} \left( \ln(2) \left(1-\frac{1}{n}\right)^{-\ln 2}
    \left( 1 - \frac{1}{2n} \right)
    + \left(1 - \frac{2}{n}\right)^{1-\ln 2} \left( 1 - \ln 2\right) \right).
    \end{align*}
    Since the second factor is an increasing function in $n$, 
    we get an upper bound when we replace it by its limit, as $n\rightarrow \infty$.
    But this limit is $1$, thus the result is proved.
\end{proof}

\begin{lemma}
    The asymptotic behavior of $\delta_{n,s}$ is
    \[
    \delta_{n,s} \sim (k-1)^s 
    \left( 1 - \frac{s((s+1)k-4)}{2n(k-1)} + O\left( \frac{s^4}{n^2}\right) \right),
    \qquad\text{as } n\to\infty. 
    \]
\end{lemma}
\begin{proof}
    In the definition of $\delta_{n,s}$ given in Equation~\eqref{delta}, we have $1\leq s\leq n -
    (n-1)/(k-1)$. A consequence of this is that all the factorials in \eqref{delta} tend to infinity 
    as $n$ tends to infinity. Thus we can use Stirling's formula and get 
    \[
    n! \underset{n\rightarrow \infty}= \sqrt{2\pi n} \left( \frac{n}{e} \right)^n 
    \left( 1 + \frac{1}{12 \; n} + \frac{1}{288 \; n^2}  + O\left(\frac{1}{n^3}\right) \right),
    \]
    and the result follows.
\end{proof}

We are now ready to define a new sequence $(g_n)_n$ such that
\begin{equation*} 
h_n = g_n + \sum_{s=1}^n \frac{(\ln 2)^s}{s!}
\left( 1 - \frac{s((s+1)k-4)}{2n(k-1)}\right) h_{n-s}.
\end{equation*}
Analogously to the binary case we obtain
\begin{equation*}
2n h_n = n g_n + n \sum_{s=0}^n \frac{(\ln 2)^s}{s!} h_{n-s}
- \frac{1}{2(k-1)} \sum_{s=0}^n \frac{(\ln 2)^s}{s!} s((s+1)k-4) h_{n-s}.
\end{equation*}
Thus we can translate this sequence into a differential equation for its generating function,
which after simplifications reads as 
\[
\left( 2 - 2^{z} \right) h'(z) 
+ \frac{ k \ln (2) z - 2}{2(k-1)} \ln (2) 2^z h(z)
= g'(z).
\]
After resolution we prove that there exist constants $\kappa, \kappa'$ such that
\begin{align*}
h(z)  \underset{z\rightarrow 1^-}\sim \kappa  \left( 2-2^z \right)^{\frac{k \ln (2) -2}{2(k-1)}}
\underset{z\rightarrow 1^-}\sim \kappa'  \left( 1-z \right)^{\frac{k \ln (2) -2}{2(k-1)}}.
\end{align*}
Through a Tauberian theorem we obtain
\[
h_n\simn K n^{\frac{2 - k \ln (2)}{2(k-1)}-1}.
\]
So, we get the following result after all. 

\bth
The number of $k$-ary weakly increasing trees with repetitions which have size $n$ is asymptotically
given by 
\[
G_n
\begin{cases}
=0 & \text{ if } n\not\equiv 1 \mod k-1, \\
\simn \eta_k \; m^{\frac{2 - k \ln (2)}{2(k-1)}} \(\frac{k-1}{\ln2}\)^m (m-1)! & \text{ if } n=1+(k-1)m.
\end{cases}
\]
\eth

\begin{figure}[htb]
    \begin{center}
        \includegraphics[width=0.7\textwidth]{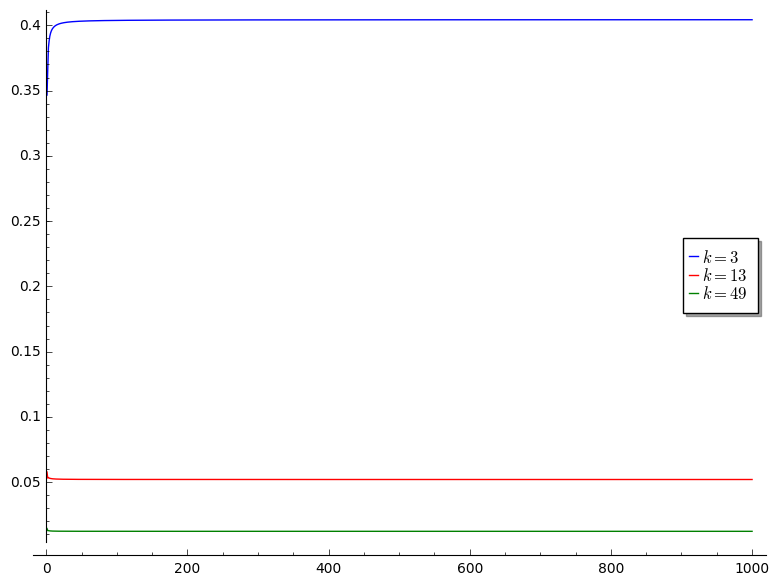}
    \end{center}
    \caption{The normalization of $B_n$ and their asymptotic behaviors for k=3, 13, 49.
        \label{fig:asympt_approx_general}}
\end{figure}
Observe that the latter theorem is coherent with the binary case after simplifications.
Furthermore for $k=3$,
the power of $m$ is approximately $0.0198...$ which can create some confusion
when we try to guess the asymptotic behavior!

\section{Conclusion}

In this article, we have shown that the asymptotic behavior of weakly
increasing binary trees of size $n$ is given by
\[
B_n \simn \eta \; n^{-\ln 2} \(\rcp{\ln 2}\)^n  (n-1)!
\]
where $\eta$ is a constant. This exhibits a certain oddity compared
to the classic asymptotic behavior of trees. In particular, the
presence of the polynomial factor $n^{-\ln(2)}$ is quite unusual and
can be compared to the classical factor $n^{-3/2}$ for the simple
family of trees.

To keep the presentation concise, we have performed only asymptotics up to the first order.
Nevertheless, the reader can easily check that the approach used can be applied to reach higher
orders. For example, for the binary case, we have this finer estimate:
\[
B_n = \eta \; n^{-\ln 2} \(\rcp{\ln 2}\)^n (n-1)! \left(1+\frac{\ln(2)}{2n}+O\left(\frac{1}{n^2}\right)\right).
\]

This refinement also provides a good idea of the speed of convergence to the
asymptotic regime. With a little more work the constant can be effectively evaluated:
$\eta\approx 0.647852\dots$.

Furthermore, we mention that our approach also allows studying characteristics of weakly
increasing trees. This gives rise to functional equations for bivariate generating functions
$f(z,u)$ and a Borel transform with respect to $z$ turns them into analytic functions in the
domain $|z|<\ln 2$ and $|u|\le 1$. Though we get formally a partial differential equation, the
singularity analysis has to be performed only with respect to $z$, which eventually leads to an
ordinary differential equation with an additional parameter. Thus we do not expect any major
problems as long as only moments are computed. For distributional results, however, certainly
uniformity of the approximations will be necessary, which might cause some technical challenges. 

The problems in the context of classical increasing trees can often be translated into the context
of urn models. There is a whole lot of literature on that topic and we refer, for example, to
the two papers that seem closest to our study. The first is by Mahmoud~\cite{Mahmoud03} and
relates urns and trees, the second by Flajolet~\emph{et al.}~\cite{FGP03} studies urns with
Analytic Combinatorics.

There is a way to encode our binary case problem as an urn model with a single color for balls
in the following way. \emph{Start with an urn with two balls. At each step, sample a subset of
    $r$ balls in the urns, and then return $2r$ balls in the urn. How may histories are there that
    give an urn containing $n$ balls?}

Unfortunately, to the best of our knowledge, urn processes with a random quantity of sampled
balls have not been studied yet. Thus it seems very promising to us to develop this new model
further.



\bibliographystyle{plain}
\bibliography{incr_trees_repet}

\end{document}